\documentclass[12pt]{amsart}
\usepackage{amssymb,amsmath,amsthm,mathrsfs,multirow,xcolor,framed,url}
\usepackage[pdftex,
         pdfauthor={Dandan Chen, Rong Chen and Frank Garvan},
         pdftitle={Congruences modulo powers of $5$ and $7$ for Ramanujan's third mock theta function},
         pdfsubject={MATHEMATICS},
         pdfkeywords={partition congruences, Dyson's rank, mock theta functions, modular functions},
         pdfproducer={Latex with hyperref},
         pdfcreator={pdflatex}]{hyperref}
\newtheorem{theorem}{Theorem}[section]

\newtheorem{prop}[theorem]{Proposition}

\theoremstyle{definition}

\theoremstyle{remark}

\numberwithin{equation}{section}

\newcommand\nutwid{\overset {\text{\lower 3pt\hbox{$\sim$}}}\nu}

\allowdisplaybreaks
\newcommand\omycite[1]{}

\newcommand{\beqs}{\begin{equation*}}
\newcommand{\eeqs}{\end{equation*}}
\newcommand{\beq}{\begin{equation}}
\newcommand{\eeq}{\end{equation}}
\renewcommand{\MR}[1]{\href{http://www.ams.org/mathscinet-getitem?mr={#1}}{MR{#1}}}
\begin{document}
%%PRELIMINARY VERSION
\title[$q$-ultraspherical polynomials]{A conjecture of Gasper on $q$-ultraspherical polynomials}

% Information for first author
\author{Dandan Chen}
\address{Department of Mathematics, Shanghai University, People's Republic of China}
\address{Newtouch Center for Mathematics of Shanghai University, Shanghai, People's Republic of China}
\email{mathcdd@shu.edu.cn}
\author{Siyu Yin}
\address{Department of Mathematics, Shanghai University, People's Republic of China}
\email{siyuyin0113@126.com}
%\thanks{}.
%    \thanks will become a 1st page footnote.

%    General info

%%\subjclass[2010]{05A19} %%, 11B65, 11F11, 11F37, 11P83, 33D15}

\date{}

%%\dedicatory{Dedicated to the memory of A.O.L. (Oliver) Atkin}


\begin{abstract}
In this paper we  establish  $q$-orthogonality relation  for the continuous $q$-ultraspherical polynomials, which was considered by Gasper. Additionally, we evaluate a new $q$-beta integral with several parameters.
\end{abstract}

\maketitle

%%%%%%%%%%%%%%%%%%%%%%%%%%%%%%%%%%%%%%%%%%%%%%%%%%%%%%%%%%%%%%%%%%%%%%%%%%%
%%\input body.tex
%SECTION 1%%%%%%%%%%%%%%%%%%%%%%%%%%%%%%%%%%%%%%%%%%%%%%%%%%%%%%%%%%%%%%%%
\section{Introduction}

L.J. Rogers introduced an extension of the ultraspherical polynomials $C_n^{\lambda}(x)$ that he used to establish the Rogers--Ramanujan identities \cite{Rogers-PLM-95}. The Rogers polynomials $C_n(x;\beta|q)$, called the continuous $q$-ultraspherical polynomials, contain the polynomials of Geronimus and are generated by the recurrence relation
\begin{align*}
(1-q^{n+1})C_{n+1}(x;\beta|q)=2x(1-\beta q^n)C_n(x;\beta|q)-(1-{\beta}^2q^{n-1})C_{n-1}(x;\beta|q),
\end{align*}
where $n\textgreater 0$, $-1\textless q\textless 1$ and
\begin{align*}
C_0(x;\beta|q)=1,\qquad C_1(x;\beta|q)=2x(1-\beta)/(1-q).
\end{align*}
When $q\rightarrow1^{-}$, $C_n(x;q^{\lambda}|q)\rightarrow C_n^{\lambda}(x)$. Rogers gave some essential properties of these polynomials but he nevertheless did not give their orthogonality. The orthogonality of these polynomials were first proved by R. Askey and M. E. H. Ismail \cite{Askey-Ismail-80} in 1980. Gasper \cite{Gasper-CM-81} showed that the more general class of functions $C_k^{(\alpha,\beta)}(e^{i\theta})$ satisfying the orthogonality relation. He also explored the $q$-(basic) analogs of $C_k^{(\alpha,\beta)}(e^{i\theta})$ which contained the continuous $q$-ultraspherical polynomials in \cite{Chen-Liu-printed} as a special case and outlined the conditions for their orthogonality. Furthermore, he thought about the more general functions $C_k^{(\alpha,\beta,\gamma)}(e^{i\theta};q)$ and $\omega^{(\alpha,\beta,\gamma)}(\theta;q)$ defined as
\begin{align*}
\sum_{k=0}^{\infty}C_k^{(\alpha,\beta,\gamma)}(e^{i\theta};q)t^k=\frac{(\alpha e^{-i\theta}t;q)_{\infty}(\beta e^{i\theta}t;q)_{\infty}}{(e^{-i\theta}t;q)_{\infty}(\gamma e^{i\theta}t;q)_{\infty}},
\end{align*}
and
\begin{align*}
\omega^{(\alpha,\beta,\gamma)}(\cos\theta;q)=\frac{(e^{-2i\theta};q)_{\infty}(\gamma e^{2i\theta};q)_{\infty}}{(\alpha e^{-2i\theta};q)_{\infty}(\beta e^{2i\theta};q)_{\infty}}.
\end{align*}
Here and later, we adopt the standard $q$-series notations
\begin{align*}
(a;q)_0:=1, \quad (a;q)_n:=\prod_{k=0}^{n-1}(1-aq^k), \quad (a;q)_\infty:=\prod_{k=0}^\infty(1-aq^k)
\end{align*}
for $|q|\textless1$ and $n \in \mathbb{N}$. Unfortunately Gasper couldn't get the similar orthogonal equation as \cite[Eq. (3.6)]{Gasper-CM-81} without additional restrictions on the parameters.

In 2006, M. E. H. Ismail and D. W. Stanton evaluated some Ramanujan continued fractions by using asymptotics of polynomials orthogonal with respect to measures with absolutely continuous components in \cite{Ismial-06}. Recently, the first author and Liu \cite{Chen-Liu-printed} gave a new proof of $q$-orthogonality relation for the continuous $q$-ultraspherical polynomials and evaluated a new $q$-beta integral with five parameters.

We consider the more general situation of $q$-orthogonal functions $C_n^{(\alpha,\beta,\gamma,\delta)}(e^{i\theta};q)$, which is defined as
\begin{align}\label{defn-cq-4}
\sum_{n=0}^{\infty}C_n^{(\alpha,\beta,\gamma,\delta)}(e^{i\theta};q)t^n
=\frac{(\alpha te^{i\theta},\beta te^{-i\theta};q)_\infty}
{(\gamma te^{i\theta}, \delta te^{-i\theta};q)_\infty},
\end{align}
where  $\alpha, \beta, \gamma, \delta \in \mathbb{C}$.
And we establish the orthogonality relation below for the $q$-functions $C_n^{(\alpha,\beta,\gamma,\delta)}(e^{i\theta};q)$.
\begin{theorem}\label{thm-5-4}
For $\max\{|q|,|\frac{\alpha}{\gamma}|,|\frac{\beta}{\delta}|\}<1$, we have
\begin{align}\label{eq-cq-orth-abcd}
&\int_{0}^{2\pi}C_m^{(\alpha,\beta,\gamma,\delta)}(e^{i\theta};q)C_n^{(\alpha,\beta,\gamma,\delta)}(e^{i\theta};q)
\omega^{(\alpha,\beta,\gamma,\delta)}(\cos\theta|q)d\theta\\
=&\frac{2\pi(\frac{\alpha}{\gamma},\frac{\beta}{\delta};q)_\infty}{(q,\frac{\alpha\beta}{\gamma\delta};q)_\infty}
\left(\frac{1}{1-\frac{\alpha q^n}{\gamma}}+\frac{1}{1-\frac{\beta q^n}{\delta}}\right)\frac{(\frac{\alpha\beta}{\gamma\delta};q)_n(\gamma\delta)^n}{(q;q)_n}\delta_{m,n},\nonumber
\end{align}
where
\begin{align}\label{w-4}
\omega^{(\alpha,\beta,\gamma,\delta)}(\cos\theta|q)=\frac{(\frac{\gamma}{\delta}e^{2i\theta},\frac{\delta}{\gamma}e^{-2i\theta};q)_\infty}{( \frac{\alpha}{\delta} e^{2i\theta}, \frac{\beta}{\gamma} e^{-2i\theta};q)_\infty}
\end{align}
and $\delta_{m,n}=1$ for $m=n$; $\delta_{m,n}=0$ for else.
\end{theorem}

Letting $\delta=1$ in Theorem \ref{thm-5-4}, we find that Theorem \ref{thm-5-4} reduces to the orthogonality which Gasper \cite{Gasper-CM-81} considered about $C_k^{(\alpha,\beta,\gamma)}(e^{i\theta};q)$ before.  Based on Theorem \ref{thm-5-4}, we derive the following new $q$-beta integral formula with seven parameters $\alpha$, $\beta$, $\gamma$, $\delta$, $s$, $t$, and $q$. More formulas about $q$-beta integral can be seen in \cite{AlSalam-Verma-PAMS-82}.

\begin{theorem}\label{q-beta-int-seven}
For $\max\{|q|,|\frac{\alpha}{\gamma}|,|\frac{\beta}{\delta}|,|\gamma s|,|\gamma t|,|\delta s|,|\delta t|\}<1$, we have
\begin{align*}
&\int_{0}^{2\pi}\frac{(\alpha t e^{i\theta},\beta t e^{-i\theta},\alpha s e^{i\theta},\beta s e^{-i\theta}, \frac{\gamma}{\delta}e^{2i\theta}, \frac{\delta}{\gamma} e^{-2i\theta};q)_\infty}
{(\gamma t e^{i\theta}, \delta t e^{-i\theta},\gamma s e^{i\theta},\delta s e^{-i\theta}, \frac{\alpha}{\delta} e^{2i\theta}, \frac{\beta}{\gamma} e^{-2i\theta};q)_\infty}d\theta\nonumber\\
=&\frac{2\pi(\frac{\alpha}{\gamma},\frac{\beta}{\delta};q)_\infty}{(q,\frac{\alpha\beta}{\gamma\delta};q)_\infty}
\sum_{n=0}^{\infty}\left(\frac{1}{1-\frac{\alpha}{\gamma}q^n}+\frac{1}{1-\frac{\beta}{\delta}q^n}\right)\frac{(\frac{\alpha\beta}{\gamma\delta};q)_n}{(q;q)_n}(\gamma\delta st)^n.
\end{align*}
\end{theorem}

Besides, letting $\alpha=a\gamma$ and $\beta=a\delta$ in Theorem \ref{thm-5-4}, we obtain the following $q$-integral formula by using the orthogonality relation for the continuous $q$-ultraspherical polynomials and Rogers $_6\phi_5$ summation formula.
\begin{theorem}\label{application-0-pi}
For $\max\{|q|,|a|,|b|\}<1$, we have
\begin{align*}
&\int_{0}^{\pi}C_m^{(b\gamma,b\delta,\gamma,\delta)}(e^{i\theta};q)C_n^{(a\gamma,a\delta,\gamma,\delta)}(e^{i\theta};q)\omega^{(a\gamma,a\delta,\gamma,\delta)}(\cos\theta|q)d\theta\\
=&\left\{
    \begin{array}{ll}
    \frac{(\gamma\delta)^n(1-aq^n)(\frac{b}{a};q)_{\frac{m-n}{2}}(b;q)_{\frac{m+n}{2}}(a\gamma\delta)^{\frac{m-n}{2}}}{(1-a)h_n(a|q)(q;q)_{\frac{m-n}{2}}(aq;q)_{\frac{m+n}{2}}}, &\text{if} ~ m\equiv n\pmod 2, \nonumber\\
    0,&\text{if}~ m \not\equiv n\pmod 2,\nonumber\\
    \end{array}
    \right.
\end{align*}
where
\begin{align*}
h_n(a|q)=\frac{(q,a^2;q)_{\infty}(q;q)_n(1-aq^n)}{2\pi(a,aq;q)_{\infty}(a^2;q)_n(1-a)}.
\end{align*}
\end{theorem}

The paper is organized as follows. In Section \ref{sec-pre}, we first collect some useful identities on basic hypergeometric series and the asymptotic properties of the $q$-functions $C_n^{(\alpha,\beta,\gamma,\delta)}(e^{i\theta};q)$. Then in Section \ref{proof-thm} we demonstrate the proofs of Theorems \ref{thm-5-4}--\ref{application-0-pi} by using the orthogonality relation for the continuous $q$-ultraspherical polynomials and Rogers $_6\phi_5$ summation formula, and we also obtain the recurrence relation between $C_n^{(a\gamma,a\delta,\gamma,\delta)}(e^{i\theta};q)$ and $C_m^{(b\gamma,b\delta,\gamma,\delta)}(e^{i\theta};q)$ without the details.

\section{Preliminaries}\label{sec-pre}
In this section, we first collect some useful identities on basic hypergeometric series. The basic hypergeometric series or the $q$-hypergeometric series $_{r+1}\phi_{r}$ are defined by \cite{Gasper-Rahman-04}
\begin{align*}
_{r+1}\phi_{r}\!\left[\begin{array}{c}
a_1,a_2,\ldots,a_{r+1}\\
b_1,b_2,\ldots,b_{r}
\end{array};q,\,z
\right]
=\sum_{k=0}^{\infty}\frac{(a_1,a_2,\ldots,a_{r+1};q)_k z^k}
{(q,b_1,\ldots,b_r;q)_k}.
\end{align*}

We will also need the $_{r+1}W_r(a_1;a_2,\cdots,a_{r-1};q,z)$ very-well-poised series:
\begin{align*}
_{r+1}W_r(a_1;a_2,\cdots,a_{r-1};q,z)=_{r+1}\phi_{r}\!\left[\begin{array}{c}
a_1,q\sqrt{a_1},-q\sqrt{a_1},a_2\ldots,a_{r-1}\\
\sqrt{a_1},-\sqrt{a_1},qa_1/a_2,\ldots,qa_1/a_{r-1}
\end{array};q,\,z
\right].
\end{align*}

The $q$-binomial theorem \cite[p. 92]{Gasper-Rahman-04} is defined as
\begin{align*}
\frac{(az;q)_\infty}{(z;q)_\infty}
=\sum_{n=0}^{\infty}\frac{(a;q)_n}{(q;q)_n}z^n, |z|<1.
\end{align*}

\begin{prop}\label{asymptotic properties}
For $\max\{|\frac{\alpha}{\gamma}|,|\frac{\beta}{\delta}|,|q|\}<1$, the $q$-function $C_n^{(\alpha,\beta,\gamma,\delta)}(e^{i\theta};q)$ has some asymptotic properties:
\begin{align*}
&(1)~|C_n^{(\alpha,\beta,\gamma,\delta)}(e^{i\theta};q)|\leq C_n^{(\alpha,\beta,\gamma,\delta)}(1;q),\\
&(2)~\limsup\limits_{n\rightarrow\infty}C_n^{(\alpha,\beta,\gamma,\delta)}(1;q)^{1/n}=\max\{|\gamma|,|\delta|\}.
\end{align*}
\end{prop}
\begin{proof}
(1)It is easy to know that $(q;q)_n>0$ when $|q|<1$, $n \in \mathbb{N}$ and $|e^{i(n-2k)\theta}|=1$. For $|\frac{\alpha}{\gamma}|<1$ and $|\frac{\beta}{\delta}|<1$, we  have $|C_n^{(\alpha,\beta,\gamma,\delta)}(e^{i\theta};q)|\leq C_n^{(\alpha,\beta,\gamma,\delta)}(1;q)$.

(2)Let $x=y=1$ in the generating function of $\Phi_n^{(\alpha,\beta,\gamma,\delta)}(x,y|q)$, then we get
\begin{align}
\frac{(\alpha t,\beta t;q)_\infty}{(\gamma t,\delta t;q)_\infty}=\sum_{n=0}^{\infty}C_n^{(\alpha,\beta,\gamma,\delta)}(1;q)t^n.
\end{align}
By the definition of the radius of convergence of power series centered at the origin, we get the radius of convergence of the above series is $\min\{\frac{1}{|\gamma|},\frac{1}{|\delta|}\}$, thus we finish the proof of (2).
\end{proof}

\begin{prop}\label{prop-cq-condi}
Let $\max\{|q|,|\frac{\alpha}{\gamma}|, |\frac{\beta}{\delta}|\}<1$. Then for any nonnegative integer $k$, and any $t$ such that $|t|<\min\{\frac{1}{|\gamma|},\frac{1}{|\delta|}\}$, the series
\begin{align*}
\sum_{n=0}^{\infty}C_{n+k}^{(\alpha,\beta,\gamma,\delta)}(e^{i\theta};q)C_{n}^{(\alpha,\beta,\gamma,\delta)}(e^{i\theta};q)
\frac{(q;q)_{n+k}}{(\frac{\alpha\beta}{\gamma\delta};q)_{n+k}}t^n
\end{align*}
converges uniformly and absolutely on $[0,2\pi]$.
\end{prop}

\begin{proof}
For $\max\{|q|,|\frac{\alpha}{\gamma}|,|\frac{\beta}{\delta}|\}<1$, with the help of the triangular inequality we can find that
\begin{align*}
\left|\frac{(q;q)_{n+k}}{(\frac{\alpha\beta}{\gamma\delta};q)_{n+k}}\right|\leq\frac{(-|q|;|q|)_\infty}{(|\frac{\alpha\beta}{\gamma\delta}|;|q|)_{n+k}}          \leq\frac{(-|q|;|q|)_\infty}{(|\frac{\alpha\beta}{\gamma\delta}|;|q|)_\infty}.
\end{align*}
From Proposition \ref{asymptotic properties}, we can get the inequality
\begin{align*}
&\left|\sum_{n=0}^{\infty}C_{n+k}^{(\alpha,\beta,\gamma,\delta)}(e^{i\theta};q)C_{n}^{(\alpha,\beta,\gamma,\delta)}(e^{i\theta};q)
\frac{(q;q)_{n+k}}{(\alpha\beta;q)_{n+k}}t^n\right|\\
\leq&\frac{(-|q|;|q|)_\infty}{(|\frac{\alpha\beta}{\gamma\delta}|;|q|)_\infty}\sum_{n=0}^{\infty}C_{n+k}^{(\alpha,\beta,\gamma,\delta)}(1;q)C_{n}^{(\alpha,\beta,\gamma,\delta)}(1;q)|t|^n.
\end{align*}
By the root test, the right-hand side of the inequality is convergent when $|t|<\min\{\frac{1}{|\gamma|},\frac{1}{|\delta|}\}$. Thus, the series in Proposition \ref{prop-cq-condi}  converges uniformly and absolutely on $[0,2\pi]$.
\end{proof}

\section{Proofs of Theorems \ref{thm-5-4} -- \ref{application-0-pi}}\label{proof-thm}
%In this section, we will present proofs of Theorem \ref{thm-5-4}, Theorem \ref{q-beta-int-seven} and Theorem \ref{application-0-pi}.

Firstly, we shall give the proof of Theorem \ref{thm-5-4}  based on the classical integral method by G. Gasper.

\begin{proof}[Proof of Theorem \ref{thm-5-4}]
We begin with expanding the right-hand side of the equation \eqref{defn-cq-4} using the $q$-binomial theorem.
\begin{align*}
\sum_{n=0}^{\infty}C_n^{(\alpha,\beta,\gamma,\delta)}(e^{i\theta};q)t^n=&\sum_{n=0}^{\infty}\frac{(\frac{\beta}{\delta};q)_n}{(q;q)_n}(\delta e^{-i\theta} t)^n \sum_{m=0}^{\infty}\frac{(\frac{\alpha}{\gamma};q)_m}{(q;q)_m}(\gamma e^{i\theta} t)^m\\
=&\sum_{k=0}^{\infty} t^k (e^{-i\theta}\delta)^k\sum_{m=0}^{k}\frac{ (\frac{\beta}{\delta};q)_{k-m} (\frac{\alpha}{\gamma};q)_m}{(q;q)_{k-m}(q;q)_m}(\frac{\gamma e^{2i\theta}}{\delta})^m\\
=&\sum_{k=0}^{\infty}t^k (e^{-i\theta}\delta)^k\frac{(\frac{\beta}{\delta};q)_k}{(q;q)_k} {_{2}\phi_{1}}\!\left[\begin{array}{c}
\frac{\alpha}{\gamma},q^{-k}\\
\frac{\delta}{\beta}q^{1-k}
\end{array};q,\,\frac{\gamma q e^{2i\theta}}{\beta}
\right].
\end{align*}
Thus, we obtain the concise form of $C_{k}^{(\alpha,\beta,\gamma,\delta)}(e^{i\theta};q)$ as follows:
\begin{align*}
C_{k}^{(\alpha,\beta,\gamma,\delta)}(e^{i\theta};q)&=(e^{-i\theta}\delta)^k\sum_{m=0}^{k}\frac{ (\frac{\beta}{\delta};q)_{k-m} (\frac{\alpha}{\gamma};q)_m}{(q;q)_{k-m}(q;q)_m}(\frac{\gamma e^{2i\theta}}{\delta})^m\\
&=(e^{-i\theta}\delta)^k\frac{(\frac{\beta}{\delta};q)_k}{(q;q)_k} {_{2}\phi_{1}}\!\left[\begin{array}{c}
\frac{\alpha}{\gamma},q^{-k}\\
\frac{\delta}{\beta}q^{1-k}
\end{array};q,\,\frac{\gamma q e^{2i\theta}}{\beta}
\right].
\end{align*}

Firstly, we calculate the result of the integral $\int_{0}^{2\pi} e^{in\theta}\omega^{(\alpha,\beta,\gamma,\delta)}(\cos\theta|q)d\theta$, where $\omega^{(\alpha,\beta,\gamma,\delta)}(\cos\theta|q)$ is defined as \eqref{w-4}.
\begin{align}\label{int-i-w}
&\int_{0}^{2\pi} e^{in\theta}\frac{(\frac{\gamma}{\delta}e^{2i\theta},\frac{\delta}{\gamma}e^{-2i\theta};q)_{\infty}}{(\frac{\alpha}{\delta}e^{2i\theta},\frac{\beta}{\gamma}e^{-2i\theta};q)_{\infty}}    d\theta\\
=&\sum_{t=0}^{\infty}\frac{(\frac{\gamma}{\alpha};q)_t}{(q;q)_t}\left(\frac{\alpha}{\delta}\right)^t \sum_{k=0}^{\infty}\frac{(\frac{\delta}{\beta};q)_k}{(q;q)_k}\left(\frac{\beta}{\gamma}\right)^k
\int_{0}^{2\pi}e^{i(n+2t-2k)\theta}.\nonumber
\end{align}
Given that $\int_{0}^{2\pi}e^{in\theta} d\theta=2\pi\delta_{n,0}$, the formula above equals $0$ when $n$ is odd. By applying Heine's transformation formula twice, \eqref{int-i-w} equals
\begin{align*}
&2\pi\frac{(\frac{\delta}{\beta};q)_m}{(q;q)_m}\left(\frac{\beta}{\gamma}\right)^m {_{2}\phi_{1}}\!\left[\begin{array}{c}
\frac{\gamma}{\alpha},\frac{\delta}{\beta}q^{m}\\
q^{m+1}
\end{array};q,\,\frac{\alpha\beta}{\gamma\delta}
\right]\\
=&2\pi\frac{(\frac{\delta}{\beta};q)_m}{(\frac{\alpha q}{\gamma};q)_m}\left(\frac{\beta}{\gamma}\right)^m \frac{(\frac{\beta}{\delta},\frac{\alpha}{\gamma};q)_{\infty}}{(\frac{\alpha\beta}{\gamma\delta},q;q)_{\infty}}(\frac{1}{1-\frac{\alpha}{\gamma}}+\frac{q^m}{1-\frac{\beta}{\delta}}).
\end{align*}
when $n=2m$, $m\in\mathbb{Z}$. Thus, it can be shown that for any integer $n$
\begin{align}\label{i-omega-int}
&\int_{0}^{2\pi} e^{in\theta}\omega^{(\alpha,\beta,\gamma,\delta)}(\cos\theta|q)d\theta
=\left\{
    \begin{array}{ll}
    0, &n \quad \text{odd}, \\
    h_{\frac{n}{2}}^{(\alpha,\beta,\gamma,\delta)} (q),&n \quad \text{even},\\
    \end{array}
    \right.\\
\text{where}~~~~~ & h_{m}^{(\alpha,\beta,\gamma,\delta)}(q) \nonumber =2\pi\frac{(\frac{\delta}{\beta};q)_m}{(\frac{\alpha q}{\gamma};q)_m}(\frac{\beta}{\gamma})^m \frac{(\frac{\beta}{\delta},\frac{\alpha}{\gamma};q)_{\infty}}{(\frac{\alpha\beta}{\gamma\delta},q;q)_{\infty}}(\frac{1}{1-\frac{\alpha}{\gamma}}+\frac{q^m}{1-\frac{\beta}{\delta}}).
\end{align}

Secondly, we deal with the integral: $\int_{0}^{2\pi} e^{in\theta}C_{k}^{(\alpha,\beta,\gamma,\delta)}(e^{i\theta};q)\omega^{(\alpha,\beta,\gamma,\delta)}(\cos\theta|q)d\theta$.
\begin{align*}
&\int_{0}^{2\pi} e^{in\theta}C_{k}^{(\alpha,\beta,\gamma,\delta)}(e^{i\theta};q)\omega^{(\alpha,\beta,\gamma,\delta)}(\cos\theta|q)d\theta\\
=&{\delta}^k\frac{(\frac{\beta}{\delta};q)_k}{(q;q)_k}\sum_{t=0}^k \frac{(\frac{\alpha}{\gamma},q^{-k};q)_t}{(q,\frac{\delta}{\beta}q^{1-k};q)_t}\left(\frac{\gamma q}{\beta}\right)^t\int_{0}^{2\pi} e^{i(n-k+2t)\theta}\omega^{(\alpha,\beta,\gamma,\delta)}(\cos\theta|q)d\theta.
\end{align*}
 By using equation \eqref{i-omega-int}, we obtain that $\int_{0}^{2\pi} e^{in\theta}C_{k}^{(\alpha,\beta,\gamma,\delta)}(e^{i\theta};q)\omega^{(\alpha,\beta,\gamma,\delta)}(\cos\theta|q)d\theta$
 equals zero when $k-n$ is odd and it equals
\begin{align*}
 2\pi & \frac{(\frac{\alpha}{\gamma},\frac{\beta}{\delta};q)_{
 \infty}(\frac{\beta}{\delta};q)_{k}{\delta}^k{\gamma}^m(\frac{\alpha q^{1-m}}{\gamma};q)_m}{ (\frac{\alpha\beta}{\gamma\delta},q;q)_{
 \infty}(q;q)_{k}{\beta}^m(\frac{\delta q^{-m}}{\beta};q)_m      } \\
&
 \times\bigg\{{\frac{1}{1-\frac{\alpha}{\gamma}} {_{3}\phi_{2}}\!\left[\begin{array}{c}
q^{-k},\frac{\alpha}{\gamma},\frac{\delta}{\beta}q^{-m}\\
\frac{\alpha q^{1-m}}{\gamma},\frac{\delta}{\beta}q^{1-k}
\end{array};q,\,q
\right]+\frac{q^{-m}}{1-\frac{\beta}{\delta}} {_{3}\phi_{2}}\!\left[\begin{array}{c}
q^{-k},\frac{\alpha}{\gamma},\frac{\delta}{\beta}q^{-m}\\
\frac{\alpha q^{1-m}}{\gamma},\frac{\delta}{\beta}q^{1-k}
\end{array};q,\,q^2
\right]\bigg\}} \nonumber
 \end{align*}
when $k-n=2m$ is even. From \cite[Eqs. (8) and (15)]{Gasper-SJMA-1981}, it shows that both of the $_{3}\phi_{2}$ series equal zero when $k>|n|$. Then we consider the result of the case: $k=|n|$. Writing the sum in the braces as $_{4}\phi_{3}$ and applying the transformation formula in \cite{Askey-Ismail-83}, we obtain that
\begin{align*}
\int_{0}^{2\pi} e^{in\theta}C_{k}^{(\alpha,\beta,\gamma,\delta)}(e^{i\theta};q)\omega^{(\alpha,\beta,\gamma,\delta)}(\cos\theta|q)d\theta
=\left\{
    \begin{array}{ll}
    \frac{2\pi(\frac{\beta}{\delta},\frac{\alpha q}{\gamma};q)_{\infty}(\frac{\alpha\beta}{\gamma\delta};q)_{k}{\delta}^k}
    {(\frac{\alpha\beta}{\gamma\delta},q;q)_{\infty}(\frac{\alpha q}{\gamma};q)_k}, &n=k>0, \\
    \frac{2\pi(\frac{\beta q}{\delta},\frac{\alpha}{\gamma};q)_{\infty}(\frac{\alpha\beta}{\gamma\delta};q)_{k}{\gamma}^k}
    {(\frac{\alpha\beta}{\gamma\delta},q;q)_{\infty}(\frac{\beta q}{\delta};q)_k},&n=-k<0.
    \end{array}
    \right.
\end{align*}
Then we can get
\begin{align*}
\int_{0}^{2\pi} \lbrace{C_{n}^{(\alpha,\beta,\gamma,\delta)}(e^{i\theta};q)\rbrace}^2\omega^{(\alpha,\beta,\gamma,\delta)}(\cos\theta|q)d\theta
=2\pi\frac{(\frac{\alpha}{\gamma},\frac{\beta}{\delta};q)_{\infty}(\frac{\alpha\beta}{\gamma\delta};q)_n(\gamma\delta)^n}{(\frac{\alpha\beta}{\gamma\delta},q;q)_{\infty}(q;q)_n}
\left(\frac{1}{1-\frac{\alpha}{\gamma}q^n}+\frac{1}{1-\frac{\beta}{\delta}q^n}\right).
\end{align*}
Hence we complete the proof of the theorem.
\end{proof}

Next,we demonstrate the proof of Theorem \ref{q-beta-int-seven} by using Theorem \ref{thm-5-4}.

\begin{proof}[Proof of Theorem \ref{q-beta-int-seven}]
Multiplying both sides of \eqref{eq-cq-orth-abcd} by $s^mt^n$, then we have
\begin{align*}
&\int_{0}^{2\pi}C_{m}^{(\alpha,\beta,\gamma,\delta)}(e^{i\theta};q)C_{n}^{(\alpha,\beta,\gamma,\delta)}(e^{i\theta};q)
\frac{(\frac{\gamma}{\delta}e^{2i\theta},\frac{\delta}{\gamma}e^{-2i\theta};q)_\infty}{(\frac{\alpha}{\delta} e^{2i\theta},\frac{\beta}{\gamma} e^{-2i\theta};q)_\infty}s^mt^nd\theta\\
=&2\pi\frac{(\frac{\alpha}{\gamma},\frac{\beta}{\delta};q)_\infty}{(q,\frac{\alpha\beta}{\gamma\delta};q)_\infty}
\left(\frac{1}{1-\frac{\alpha}{\gamma} q^n}+\frac{1}{1-\frac{\beta}{\delta} q^n}\right)\frac{(\frac{\alpha\beta}{\gamma\delta};q)_n(\gamma\delta)^n}{(q;q)_n}s^mt^n\delta_{m,n}.
\end{align*}
Summing the above equation about $m$ and $n$ from $m=0$ to $m=\infty$ and $n=0$ to $n=\infty$ respectively, we deduce that
\begin{align*}
&\int_{0}^{2\pi}\left(\sum_{m=0}^{\infty}C_{m}^{(\alpha,\beta,\gamma,\delta)}(e^{i\theta};q)s^m\right)
\left(\sum_{n=0}^{\infty}C_{n}^{(\alpha,\beta,\gamma,\delta)}(e^{i\theta};q)t^n\right)
\frac{(\frac{\gamma}{\delta}e^{2i\theta},\frac{\delta}{\gamma}e^{-2i\theta};q)_\infty}{(\frac{\alpha}{\delta} e^{2i\theta},\frac{\beta}{\gamma} e^{-2i\theta};q)_\infty}d\theta\\
=&2\pi\frac{(\frac{\alpha}{\gamma},\frac{\beta}{\delta};q)_\infty}{(q,\frac{\alpha\beta}{\gamma\delta};q)_\infty}
\sum_{n=0}^{\infty}\left(\frac{1}{1-\frac{\alpha}{\gamma} q^n}+\frac{1}{1-\frac{\beta}{\delta} q^n}\right)\frac{(\frac{\alpha\beta}{\gamma\delta};q)_n}{(q;q)_n}(\gamma\delta st)^n.
\end{align*}
Applying \eqref{defn-cq-4} to the left-hand side of the above equation, we complete the proof of Theorem \ref{q-beta-int-seven}.
\end{proof}

Now  we shall give the proof of  Theorem \ref{application-0-pi} about the orthogonal relation to q-functions  $C_n^{(a\gamma,a\delta,\gamma,\delta)}(e^{i\theta};q)$.
\begin{proof}[Proof of Theorem \ref{application-0-pi}]
Let $\alpha=a\gamma$ and $\beta=a\delta$ in \eqref{defn-cq-4}. Then  the generating function of $C_n^{(a\gamma,a\delta,\gamma,\delta)}(e^{i\theta};q)$ is
\begin{align}\label{defn-cq-4-a}
\sum_{n=0}^{\infty}C_n^{(a\gamma,a\delta,\gamma,\delta)}(e^{i\theta};q)t^n
=\frac{(a\gamma te^{i\theta},a\delta te^{-i\theta};q)_\infty}
{(\gamma te^{i\theta}, \delta te^{-i\theta};q)_\infty}.
\end{align}
Also, the orthogonal relation is that for $\max\{|q|,|a|\}<1$,
\begin{align*}
&\int_{0}^{2\pi}C_m^{(a\gamma,a\delta,\gamma,\delta)}(e^{i\theta};q)C_n^{(a\gamma,a\delta,\gamma,\delta)}(e^{i\theta};q)\omega^{(a\gamma,a\delta,\gamma,\delta)}(\cos\theta|q)d\theta\\
=&\frac{4\pi(a,a;q)_{\infty}}{(q,a^2;q)_{\infty}}\frac{1}{1-aq^n}\frac{(a^2;q)_n(\gamma\delta)^n}{(q;q)_n}\delta_{m,n}
\end{align*}
by letting $\alpha=a\gamma$ and $\beta=a\delta$ in Theorem \ref{thm-5-4}.
By some  simplify calculations, the left-hand side of the above formula becomes
\begin{align*}
&\int_{0}^{2\pi}C_m^{(a\gamma,a\delta,\gamma,\delta)}(e^{i\theta};q)C_n^{(a\gamma,a\delta,\gamma,\delta)}(e^{i\theta};q)\omega^{(a\gamma,a\delta,\gamma,\delta)}(\cos\theta|q)d\theta\\
=&2\int_{0}^{\pi}C_m^{(a\gamma,a\delta,\gamma,\delta)}(e^{i\theta};q)C_n^{(a\gamma,a\delta,\gamma,\delta)}(e^{i\theta};q)\omega^{(a\gamma,a\delta,\gamma,\delta)}(\cos\theta|q)d\theta.
\end{align*}
Recall \cite[p. 186]{Gasper-Rahman-04}, for $|q|<1,|\beta|<1$,
\begin{align*}
\int_{0}^{\pi}C_m(\cos\theta;\beta|q)C_n(\cos\theta;\beta|q)W_{\beta}(\cos\theta|q)d\theta=\frac{\delta_{m,n}}{h_n(\beta|q)},
\end{align*}
where
\begin{align*}
h_n(\beta|q)=\frac{(q,\beta^2;q)_{\infty}(q;q)_n(1-\beta q^n)}{2\pi(\beta,\beta q;q)_{\infty}(\beta^2;q)_n(1-\beta)}.
\end{align*}
More details can be found in \cite{Askey-Ismail-83,Askey-Wilson-85,Ismail-05,KLS-10}.
Using the definition of $h_n(\beta|q)$ above, we can get the orthogonal relation, for $\max\{|q|,|a|\}<1$,
\begin{align}\label{eq-cq-orth-2}
\int_{0}^{\pi}C_m^{(a\gamma,a\delta,\gamma,\delta)}(e^{i\theta};q)C_n^{(a\gamma,a\delta,\gamma,\delta)}(e^{i\theta};q)\omega^{(a\gamma,a\delta,\gamma,\delta)}(\cos\theta|q)d\theta
=\frac{(\gamma\delta)^m\delta_{m,n}}{h_n(a|q)},
\end{align}
where
\begin{align*}
h_n(a|q)=\frac{(q,a^2;q)_{\infty}(q;q)_n(1-aq^n)}{2\pi(a,aq;q)_{\infty}(a^2;q)_n(1-a)}.
\end{align*}
By \eqref{defn-cq-4-a}, we have the following summation of $C_n^{(\alpha,\alpha,\gamma,\gamma)}(e^{i\theta};q)(1-aq^n)t^n$ about $n$ from $0$ to $\infty$ :
\begin{align}\label{eq-cq-orth-qn}
\sum_{n=0}^{\infty}C_n^{(a\gamma,a\delta,\gamma,\delta)}(e^{i\theta};q)(1-aq^n)t^n
=\frac{(a\gamma tqe^{i\theta},a\delta tqe^{-i\theta};q)_\infty}
{(\gamma te^{i\theta}, \delta te^{-i\theta};q)_\infty}(1-a)(1-a\gamma\delta t^2).
\end{align}
Combining \eqref{eq-cq-orth-2} with\eqref{eq-cq-orth-qn}, we obtain that
\begin{align*}
&\int_{0}^{\pi}C_m^{(a\gamma,a\delta,\gamma,\delta)}(e^{i\theta};q)\omega^{(a\gamma,a\delta,\gamma,\delta)}(\cos\theta|q)\frac{(a\gamma tqe^{i\theta},a\delta tqe^{-i\theta};q)_{\infty}}{(\gamma te^{i\theta},\delta te^{-i\theta};q)_{\infty}}(1-a\gamma\delta t^2)d\theta\\
=&\frac{2\pi(\gamma\delta)^m(a,aq;q)_{\infty}(a^2;q)_mt^m}{(q,a^2;q)_{\infty}(q;q)_m}.
\end{align*}
Then let $t\rightarrow tq^k$ in the equation above, we have
\begin{align}\label{eq-cq-orth-qk}
\int_{0}^{\pi}C_m^{(a\gamma,a\delta,\gamma,\delta)}(e^{i\theta};q)\omega^{(a\gamma,a\delta,\gamma,\delta)}&(\cos\theta|q)\frac{(a\gamma tqe^{i\theta},a\delta tqe^{-i\theta};q)_{\infty}}{(\gamma te^{i\theta},\delta te^{-i\theta};q)_{\infty}}
\frac{(\gamma te^{i\theta},\delta te^{-i\theta};q)_k }{(a\gamma tqe^{i\theta},a\delta tqe^{-i\theta};q)_k}\\
&\times(1-a\gamma\delta t^2q^{2k})d\theta
=\frac{2\pi(\gamma\delta)^{m}(a,aq;q)_{\infty}(a^2;q)_m(tq^k)^m}{(q,a^2;q)_{\infty}(q;q)_m}.\nonumber
\end{align}
Multiplying both sides of \eqref{eq-cq-orth-qk} by
\begin{align*}
\frac{(a\gamma\delta t^2,\frac{aq}{b};q)_k}{(1-a\gamma\delta t^2)(q,b\gamma\delta t^2;q)_k}b^k
\end{align*}
and then summing the resulting equation about $k$ from $0$ to $\infty$, we get that
\begin{align*}
&\int_{0}^{\pi}C_m^{(a\gamma,a\delta,\gamma,\delta)}(e^{i\theta};q)\omega^{(a\gamma,a\delta,\gamma,\delta)}(\cos\theta|q){_6W_5(a\gamma\delta t^2;\gamma te^{i\theta},\delta te^{-i\theta},\frac{aq}{b};q,b)}
\frac{(a\gamma tqe^{i\theta},a\delta tqe^{-i\theta};q)_{\infty}}{(\gamma te^{i\theta},\delta te^{-i\theta};q)_{\infty}}d\theta\\
=&\frac{2\pi(\gamma\delta)^{m}(a,aq;q)_{\infty}(a^2;q)_mt^m}{(q,a^2;q)_{\infty}(q;q)_m(1-a\gamma\delta t^2)}\sum_{k=0}^{\infty}\frac{(a\gamma\delta t^2,\frac{aq}{b};q)_k}{(q,b\gamma\delta t^2;q)_k}(bq^m)^k.
\end{align*}

From the Rogers $_6\phi_5$ summation formula \cite[Eq. (2.7.1)]{Gasper-Rahman-04}, for $|aq/bcd|<1$,
\begin{align*}
_6W_5(a;b,c,d;q,\frac{aq}{bcd})=\frac{(aq,aq/bc,aq/bd;q)_{\infty}}{(aq/b,aq/c,aq/d,aq/bcd;q)_{\infty}},
\end{align*}
we arrive at
\begin{align*}
_6W_5(a\gamma\delta t^2;\gamma te^{i\theta},\delta te^{-i\theta},\frac{aq}{b};q,b)
=\frac{(a\gamma\delta t^2q,aq,b\gamma te^{i\theta},b\delta te^{-i\theta};q)_{\infty}}{(a\gamma tqe^{i\theta},a\delta tqe^{-i\theta},b\gamma\delta t^2,b;q)_{\infty}}.
\end{align*}
Then we can obtain that
\begin{align*}
&\int_{0}^{\pi}C_m^{(a\gamma,a\delta,\gamma,\delta)}(e^{i\theta};q)\omega^{(a\gamma,a\delta,\gamma,\delta)}(\cos\theta|q)\frac{(b\gamma te^{i\theta},b\delta te^{-i\theta};q)_{\infty}}{(\gamma te^{i\theta},\delta te^{-i\theta};q)_{\infty}}d\theta\\
=&\frac{2\pi(\gamma\delta)^{m}t^m(a,b;q)_{\infty}(a^2;q)_m}{(q,a^2;q)_{\infty}(q;q)_m}\sum_{k=0}^{\infty} \frac{(\frac{aq}{b};q)_k(b\gamma\delta t^2q^{k};q)_{\infty}}{(q;q)_k(a\gamma\delta t^2q^k;q)_{\infty}}(bq^m)^k.
\end{align*}
Considering the summation on the right-hand side of the formula above, by $q$-binomial theorem, we have
\begin{align*}
\sum_{k=0}^{\infty} \frac{(\frac{aq}{b};q)_k(b\gamma\delta t^2q^k;q)_{\infty}}{(q;q)_k(a\gamma\delta t^2q^k;q)_{\infty}}(bq^m)^k
=&\sum_{k=0}^{\infty} \frac{(\frac{aq}{b};q)_k}{(q;q)_k}(bq^m)^k
\sum_{s=0}^{\infty}\frac{(\frac{b}{a};q)_s}{(q;q)_s}(a\gamma\delta t^2q^k)^s\\
=&\sum_{s=0}^{\infty}\frac{(\frac{b}{a};q)_s}{(q;q)_s}(a\gamma\delta t^2)^s
\sum_{k=0}^{\infty} \frac{(\frac{aq}{b};q)_k}{(q;q)_k}(bq^{m+s})^k\\
=&\sum_{s=0}^{\infty}\frac{(\frac{b}{a};q)_s}{(q;q)_s}(a\gamma\delta t^2)^s\frac{(aq^{m+s+1};q)_{\infty}}{(bq^{m+s};q)_{\infty}}\\
=&\frac{(a;q)_{\infty}}{(b;q)_{\infty}}
\sum_{s=0}^{\infty}\frac{(\frac{b}{a};q)_s(b;q)_{m+s}}{(q;q)_s(a;q)_{m+s+1}}(a\gamma\delta t^2)^s.
\end{align*}
Hence, we deduce that
\begin{align*}
\int_0^{\pi} C_m^{(a\gamma,a\delta,\gamma,\delta)}(e^{i\theta};q)\frac{(b\gamma te^{i\theta},b\delta te^{-i\theta};q)_{\infty}}{(\gamma te^{i\theta},\delta te^{-i\theta};q)_{\infty}}\omega^{(a\gamma,a\delta,\gamma,\delta)}(\cos\theta|q)d\theta\\
=\frac{2\pi(a,a;q)_{\infty}(a^2;q)_m(\gamma\delta)^{m}t^m}{(q,a^2;q)_{\infty}(q;q)_m}
\sum_{n=0}^{\infty}\frac{(\frac{b}{a};q)_n(b;q)_{m+n}}{(q;q)_n(a;q)_{m+n+1}}(a\gamma\delta t^2)^n.
\end{align*}
With the help of \eqref{defn-cq-4-a}, it follows that
\begin{align*}
\sum_{m=0}^{\infty}\left\{\int_0^{\pi}C_n^{(a\gamma,a\delta,\gamma,\delta)}(e^{i\theta};q)C_m^{(b\gamma,b\delta,\gamma,\delta)}(e^{i\theta};q)\omega^{(a\gamma,a\delta,\gamma,\delta)}(\cos\theta|q)d\theta\right\}t^m\\
=\frac{2\pi(a,a;q)_{\infty}(a^2;q)_n(\gamma\delta)^{n}}{(q,a^2;q)_{\infty}(q;q)_n}
\sum_{j=0}^{\infty}\frac{(\frac{b}{a};q)_j(b;q)_{n+j}}{(q;q)_j(a;q)_{n+j+1}}(a\gamma\delta)^j t^{n+2j}.
\end{align*}
Comparing the coefficients of powers of $t$ in the above equation, we complete the proof of Theorem \ref{application-0-pi}.
\end{proof}

Besides, we also obtain the recurrence relation between $C_n^{(a\gamma,a\delta,\gamma,\delta)}(e^{i\theta};q)$ and $C_m^{(b\gamma,b\delta,\gamma,\delta)}(e^{i\theta};q)$.
\begin{prop} For $\max\{|a|,|b|\}<1$, we have
\begin{align*}
C_m^{(b\gamma,b\delta,\gamma,\delta)}(e^{i\theta};q)=\sum_{\substack{n=0\\n\equiv m\pmod2}}^m (1-aq^n)C_n^{(a\gamma,a\delta,\gamma,\delta)}(e^{i\theta};q)\frac{(\frac{b}{a};q)_{\frac{m-n}{2}}(b;q)_{\frac{m+n}{2}}}{(q;q)_{\frac{m-n}{2}}(a;q)_{\frac{m+n}{2}+1}}(a\gamma\delta)^{\frac{m-n}{2}}.
\end{align*}
\end{prop}

\subsection*{Acknowledgements}

The first author was supported in part by  the National Natural Science Foundation of China (Grant No. 12201387).

%%%HERE

%APPENDIX A%%%%%%%%%%%%%%%%%%%%%%%%%%%%%%%%%%%%%%%%%%%%%%%%%%%%%%%%%%%%%%%
%%\input appendixA.tex

%\subsection*{Acknowledgments}
%I would like to thank XXXXXXXXXXXXXXXXXXXXXX for his comments and suggestions.

%%%%%%%%%%%%%%%%%%%%%%%%%%%%%%%%%%%%%%%%%%%%%%%%%%%%%%%%%%%%%%%%%%%%%%%%%%%

%%%%%%%%%%%%%%%%%%%%%%%%%%%%%%%%%%%%%%%%%%%%%%%%%%%%%%%%%%%%%%%%%%%%%%%%%%%%%%%

%%%%%%%%%%%%%%%%%%%%%%%%%%%%%%%%%%%%%%%%%%%%%%%%%%%%%%%%%%%%%%%%%%%%%%%%%%%%%%%

\end{document}